\documentclass[11pt]{article}
\usepackage{url,ifthen}
\usepackage{srcltx}
\usepackage{multirow}
\usepackage{boxedminipage}
\usepackage[margin=1.1in]{geometry}
\usepackage{nicefrac}
\usepackage{xspace}
\usepackage{graphicx}
\usepackage{srcltx}
\usepackage[usenames]{color}
\usepackage{fullpage}
\usepackage{xspace}
\definecolor{DarkGreen}{rgb}{0.1,0.5,0.1}
\definecolor{DarkRed}{rgb}{0.5,0.1,0.1}
\definecolor{DarkBlue}{rgb}{0.1,0.1,0.5}
\usepackage[small]{caption}
\usepackage{float}
\usepackage{balance}
\usepackage{amsmath}
\usepackage{amsfonts}
\usepackage{amssymb}
\usepackage{amsthm}

\newtheorem{theorem}{Theorem}
\newtheorem*{namedtheorem}{\theoremname}
\newcommand{\theoremname}{testing}

\newtheorem{lemma}{Lemma}

\newtheorem{claim}{Claim}

\newtheorem{corollary}[theorem]{Corollary}

\newtheorem{conjecture}{Conjecture}
\newtheorem*{question*}{Question}

\theoremstyle{definition}
\newtheorem{definition}{Definition}

\theoremstyle{plain}
\newtheorem{Alg}{Algorithm}

\definecolor{DarkGreen}{rgb}{0.1,0.5,0.1}
\definecolor{DarkRed}{rgb}{0.5,0.1,0.1}
\definecolor{DarkBlue}{rgb}{0.1,0.1,0.5}

\usepackage[pdftex]{hyperref}
\hypersetup{
    unicode=false,          
    pdftoolbar=true,        
    pdfmenubar=true,        
    pdffitwindow=false,      
    pdfnewwindow=true,      
    colorlinks=true,       
    linkcolor=DarkGreen,          
    citecolor=DarkGreen,        
    filecolor=DarkGreen,      
    urlcolor=DarkBlue,          
}

\newcommand{\ignore}[1]{}

\newcommand{\R}{\mathbb R}

\makeatletter
\floatstyle{ruled}
\newfloat{fragment}{H}{lop}
\floatname{fragment}{Algorithm}
\renewcommand{\floatc@ruled}[2]{\vspace{2pt}{\@fs@cfont \#1.\:} \#2 \par
 \vspace{1pt}}
\makeatother

\title{The Paulsen Problem Made Simple\footnote{\textbf{The results in this full paper were briefly announced in the conference {\em ITCS 2019} but without any accompanying proofs.}}}
\author{Linus Hamilton\thanks{Massachusetts Institute of Technology. Department of Mathematics. Email: {\tt luh@mit.edu}. This work was supported in part by a Fannie and John Hertz Foundation Fellowship.} \and Ankur Moitra\thanks{
Massachusetts Institute of Technology. Department of Mathematics and the Computer Science and Artificial Intelligence Lab. Email: {\tt moitra@mit.edu}. This work was supported in part by NSF CAREER Award CCF-1453261, NSF Large CCF-1565235, a David and Lucile Packard Fellowship and an ONR Young Investigator Award.} }

\begin{document}
\maketitle

\begin{abstract}
\normalsize
The Paulsen problem was once thought to be one of the most intractable problems in frame theory. It was resolved in a recent tour-de-force work of Kwok, Lau, Lee and Ramachandran. In particular, they showed that every $\epsilon$-nearly equal norm Parseval frame in $d$ dimensions is within squared distance $O(\epsilon d^{13/2})$ of an equal norm Parseval frame. We give a dramatically simpler  proof based on the notion of radial isotropic position, and along the way show an improved bound of $O(\epsilon d^2)$. 
\end{abstract}

\thispagestyle{empty}

\newpage

\setcounter{page}{1}

\section{Introduction}

The Paulsen problem was once thought to be one of the most intractable problems in frame theory \cite{CC, C}. 
To state the problem, we need the following definition:

\begin{definition}
We say that a set of vectors $v_1, v_2, \dots, v_n \in \R^d$ is an {\em equal norm Parseval frame} if 
$$\sum_{i =1}^n v_i v_i^T = I \mbox{ and } \|v_i\|_2^2 = \frac{d}{n} \mbox{ for each }i$$
Alternatively, we say that it is an {\em $\epsilon$-nearly equal norm Parseval frame} if
$$(1-\epsilon) I \preceq  \sum_{i =1}^n v_i v_i^T \preceq (1+\epsilon) I \mbox{ and } (1-\epsilon) \frac{d}{n} \leq \|v_i\|_2^2 \leq  (1+\epsilon) \frac{d}{n} \mbox{ for each }i$$
\end{definition}

When we drop the condition on the norm of each vector, we refer to the set of vectors as a {\em Parseval frame} or an {\em $\epsilon$-nearly Parseval frame} respectively. Let $\mathcal{F}$ denote the set of all equal norm Parseval frames. Lastly for two sequences of vectors $V = v_1, v_2, \dots, v_n$ and $W = w_1, w_2, \dots, w_n$ of the same length, we let
$$\mbox{dist}^2(V, W) = \sum_{i=1}^n \|v_i - w_i \|^2$$
With this terminology in hand, the Paulsen problem asks:

\begin{conjecture}
For every $\epsilon$-nearly equal norm Parseval frame $V$, is
$$\inf_{W \in \mathcal{F}} \mbox{{\em dist}}^2(V, W)$$
bounded by a fixed polynomial in $\epsilon$ and $d$?
\end{conjecture}

Early results gave bounds on the squared distance that were polynomial in $\epsilon$, $d$ and $n$ \cite{CC, CFM}. Cahill and Casazza \cite{CC} showed that (up to constant factors) it is equivalent to another fundamental problem in operator theory called the {\em projection problem}: 

\begin{conjecture} Let $e_1, e_2, \cdots, e_n$ be the standard basis vectors in $\R^n$. Given a projection $P$ onto a $d$-dimensional subspace that satisfies 
$$(1-\epsilon) \frac{d}{n} \leq \| Pe_i\|_2^2 \leq (1+\epsilon) \frac{d}{n} \mbox{ for each } i$$
is there a projection $Q$ with $\|Q e_i \|_2^2 = \frac{d}{n}$ and
$$\sum_{i=1}^n \|Pe_i - Qe_i\|_2^2$$
bounded by a fixed polynomial in $\epsilon$ and $d$?
\end{conjecture}

After fifteen years of being listed as a major open problem in frame theory \cite{C, M, CL}, Kwok, Lau, Lee and Ramachandran \cite{KLLR} gave the first bound that was polynomial in $\epsilon$ and $d$. Through a tour-de-force utilizing operator scaling, connections to dynamical systems and ideas from smoothed analysis, they proved that the  squared distance is at most $O(\epsilon d^{13/2})$. Their paper was $104$ pages long and highly complex. Our main result is a dramatically simpler proof of the Paulsen conjecture, that also yields a much better bound:

\begin{theorem}[Main]\label{thm:main}
For any $\epsilon$-nearly equal norm Parseval frame $V$, there is an equal norm Parseval frame $W$ with $$\mbox{{\em dist}}^2(V, W) \leq 20 \epsilon d^2$$
\end{theorem}

In terms of lower bounds, Cahill and Casazza \cite{CC} gave a family of examples of $\epsilon$-nearly equal norm Parseval frames where the squared distance to the closest equal norm Parseval frame is at least $\Omega(\epsilon d)$. It is an interesting open question to close this gap. 

Our main idea is to make use of the notion of {\em radial isotropic position}\footnote{This concept goes by many other names in the literature, such as {\em well-spread vectors} \cite{DSW} or {\em geometric scaling for rank one Brascamp-Lieb datum} \cite{GGOW}. The name we use here originated in \cite{HM}.}. In the next section, we define it formally. But to understand it informally, it is useful to compare it to the more familiar notion of placing a set of vectors in isotropic position: Given a set of vectors  $V = v_1, v_2, \dots, v_n \in \R^d$, is there an invertible affine transformation that generates a new set of vectors $Y = Av_1 + b, Av_2 + b, \dots, A v_n + b$ that has mean zero and identity covariance? It is well known that there is such a transformation if and only if $\sum_i v_i v_i^T$ has full rank. 

However such a transformation can also stretch out some directions much more than others \---- e.g. if all but one of the vectors in $V$ are contained in a $d-1$-dimensional subspace. In this case, the set of vectors after applying the transformation would be quite far from where it started out, in total squared distance. Informally, radial isotropic position asks for a linear transformation $A$ so that the {\em renormalized} vectors $w_i = A v_i/ \|A v_i\|$ have the property that $\sum_i w_i w_i^T$ is a scalar multiple of the identity. The transformation is now nonlinear but is particularly well suited for constructing a close by equal norm Parseval frame. 

One can now ask the same sort of question as before: When can a set of vectors be placed in radial isotropic position? Barthe \cite{B} gave a complete characterization of when this is  and is not possible which in turn plays a key role in our proof. It turns out that a sufficient condition is that every $d$ vectors are linearly independent. Now we construct an equal norm Parseval frame as follows: First we renormalize the vectors in $V$ and then we perturb them. Perturbations play a delicate role in \cite{KLLR}. They give a dynamical system which constructs an equal norm Parseval frame from an $\epsilon$-nearly equal norm Parseval frame as its input. In order to bound the total squared distance between the input and output, they need to lower bound the convergence rate. They do this through a certain pseudorandom property (Definition $4.3.2$) which they show holds when the input is appropriately perturbed. In our proof, all we need is that the perturbations do not move the set of points by too much in squared distance and that afterwards every $d$ of them are linearly independent\footnote{In particular, essentially all sufficiently small perturbations would work for us. It could even be an infinitesimal perturbation because we do not need any quantitative bounds on how far they are from having a non-trivial linear dependence.}. The latter condition guarantees that there is a linear transformation that places them in radial isotropic position. Let $W$ be the set of vectors, after applying the linear transformation and renormalizing. By definition, it is an equal norm Parseval frame. Our main technical contribution is in bounding the squared distance between $V$ and $W$, which we do through some elementary but subtle algebraic manipulations. 

Taking a step back, the notion of radial isotropic position seems quite powerful and mysterious but has thus far only found a handful of applications. Forster \cite{F} used it to prove a remarkable lower bound in communication complexity (by lower bounding the sign rank of the Hadamard matrix). Hardt and Moitra \cite{HM} gave the first algorithm for computing the transformation that places a set of vectors in radial isotropic position (under a slight strengthening of Barthe's conditions). They also gave applications to linear regression in the presence of outliers. Dvir, Saraf and Wigderson \cite{DSW} used it to prove superquadratic lower bounds for $3$-query locally correctable codes over the reals. Here we use it to give a simple proof of the Paulsen conjecture. Are there other exciting applications waiting to be discovered? 

\subsection*{Connections to Operator Scaling and the Brascamp-Lieb Inequality}

Radial isotropic position is itself a special case of the more general notion of {\em geometric position} \cite{Ba, B} where we are given an $n$ tuple of linear transformations $B_1, B_2, \dots, B_n$ of dimensions $d_1 \times d, d_2 \times d, \dots, d_n \times d$ and a nonnegative vector $c$ of dimension $n$ with $\sum_{i=1}^n c_i d_i = d$ and the goal is to find square, invertible matrices $A_1, A_2, \dots, A_n$ and $A$ so that
$$\sum_{i=1}^n c_i \Big ( A_i^{-1} B_i A\Big)^T \Big ( A_i^{-1} B_i A\Big) = I \mbox{ and } \Big ( A_i^{-1} B_i A\Big) \Big ( A_i^{-1} B_i A\Big)^T = I \mbox{ for each } i$$
If we set $d_i =1$ for all $i$, then each linear transformation $B_i$ can be written as the inner-product with some vector $v_i$. Now if we also set $c_i = \frac{d}{n}$ for all $i$, it is easy to check that $A$ places the set of vectors $v_1, v_2, \dots, v_n$ in radial isotropic position. 

It turns out that having $A_1, A_2, \dots, A_n$ and $A$ that place $B_1, B_2, \dots, B_n$ in geometric position with respect to the vector $c$ yields an explicit expression for the best constant $C$ for which the inequality
$$\int_{x \in \R^d} \prod_{i=1}^n \Big (f_i (B_i x)\Big)^{c_i} dx \leq C \prod_{i=1}^n \Big ( \int_{x_i \in \R^{d_i}} f_i(x_i) dx_i \Big)^{c_i}$$
holds over all $m$ tuples of nonnegative functions $f_1, f_2, \dots, f_m$ \cite{BCCT}. 
This is called the Brascamp-Lieb inequality. 

Finally, in terms of how to compute $A_1, A_2, \dots, A_n$ and $A$, a popular approach is {\em operator scaling} \cite{G} and there has been considerable recent progress in bounding the number of iterations it needs \cite{GGOW1, GGOW}. As we mentioned, Kwok, Lau, Lee and Ramachandran \cite{KLLR} used operator scaling to solve the Paulsen conjecture. In this sense, our approach and theirs are closely related in that they both revolve around algorithms (in our case the ellipsoid algorithm) for computing radial isotropic position. Perhaps the main technical divergence is that they track how the squared distance changes after each iteration of operator scaling, while we are able to bound the squared distance just based on transformation that places $v_1, v_2, \dots, v_n$ into radial isotropic position. It is also worth mentioning that if instead of proving existence of a nearby equal norm Parseval frame, we want to find it up to some target precision $\delta$, the approaches based on operator scaling typically require the number of iterations to be polynomial in $1/\delta$. In contrast, we will give algorithms whose running time is polynomial in $\log 1/\delta$.

\section{Radial Isotropic Position and the Proof}

First we introduce some of the basic concepts and results about radial isotropic position. We will do so in slightly more generality than we will ultimately need. 

\begin{definition}
We say that a set of vectors $u_1, u_2, \dots, u_n \in \R^d$ is in {\em radial isotropic position} with respect to a coefficient vector $c \in \R^n$ if
$$\sum_{i=1}^n c_i \Big (\frac{u_i}{\|u_i\|} \Big) \Big (\frac{u_i}{\|u_i\|} \Big)^T = I$$
\end{definition}

Note that if we take the trace of both sides in the expression, we get the necessary condition that $\sum_{i=1}^n c_i = d$. In fact we will only ever consider the case when each $c_i = \frac{d}{n}$. We will also need the following key definition:

\begin{definition}\cite{E}\label{def:basis}
For a set $U$ of vectors $u_1, u_2, \dots, u_n \in \R^d$, its {\em basis polytope} is defined as
$$\mathcal{B}(U) = \Big \{ c \in \R^n \mbox{ s.t. } \sum_{i=1}^n c_i = d \mbox{ and for all } A \subseteq [n], \mbox{dim}\Big(\mbox{span}\{u_i\}_{i \in A}\Big) \geq \sum_{i \in A} c_i  \Big \}$$
\end{definition}

Now we are ready to state Barthe's theorem:

\begin{theorem}\cite{B}\label{thm:barthe}
A set of vectors $U = u_1, u_2, \dots, u_n \in \R^d$ can be put into radial isotropic position with respect to $c$ by a linear transformation if and only if $c \in \mathcal{B}(U)$
\end{theorem}

Some further remarks: $(1)$ The usual definition of the basis polytope is based on taking the convex hull of the indicators of subsets of vectors in $U$ that form a basis. $(2)$ The alternative definition we gave in Definition~\ref{def:basis} will be more directly useful for our purposes, and was proven to be equivalent by Edmonds \cite{E}. He used this equivalence to give a separation oracle for the basis polytope, which in turn plays a key role in the algorithm of Hardt and Moitra \cite{HM} for computing the linear transformation that puts a set of vectors into radial isotropic position. 

Now we are ready to prove our main theorem:

\begin{proof}
Let $V = v_1, v_2, \dots, v_n$. Then set $$u_i =\sqrt{\frac{d}{n}} \frac{v_i}{\|v_i\|} + \eta_i$$
where $\eta_i$ is a perturbation. What we need from these perturbations is just that they make every set of $d$ vectors in $U$ be linearly independent, and that if the norm of the perturbation is a large polynomial in $\epsilon$, $1/d$ and $1/n$ that it has a negligible effect on our squared distance bounds. As we go along, we will quantify how small we need the perturbation to be. First we want to bound the squared distance between $V$ and $U$. We can upper bound
$$\mbox{dist}^2(v_i, u_i) \leq \Big (\sqrt{\frac{d}{n}} - \sqrt{(1 - \epsilon)\frac{d}{n}} \Big )^2 + \gamma \leq \frac{\epsilon d}{n}$$
where  $\gamma \leq \|\eta_i \|^2 + 2 \|\eta_i\|$ is a term that depends on the perturbation and if $\gamma \leq \frac{(1- \sqrt{1-\epsilon} ) \epsilon d}{n}$ then the last inequality holds. Now summing over all pairs of vectors we get that $\mbox{dist}^2(V, U) \leq \epsilon d$. 

Next we observe that the vectors in $U$ are still a nearly equal norm Parseval frame. Before adding the perturbation, each vector $v_i$ was scaled by a factor between $\sqrt{1-\epsilon}$ and $\sqrt{1 + \epsilon}$. Also if we take $\|\eta_i \| \leq \frac{\epsilon}{2n}$ for each $i$, then we conclude that the vectors in $U$ are a $4\epsilon$-nearly equal norm Parseval frame. 

Now we will utilize Theorem~\ref{thm:barthe}. We work with the coefficient vector $c$ where $c_i = \frac{d}{n}$ for each $i$. It is easy to check in Definition~\ref{def:basis} that, because every set of $d$ vectors from $U$ are linearly independent, $c$ is in their basis polytope. Hence by Theorem~\ref{thm:barthe} we are guaranteed that there is a linear transformation $A$ that places them in radial isotropic position with respect to $c$. We claim that we can assume without loss of generality that $A$ is a nonnegative diagonal matrix whose entries are sorted in non-increasing order along the diagonal. This follows by writing the singular value decomposition $A = C \Sigma D^T$ and observing that radial isotropic position is preserved under taking orthogonal transformations. Thus $A' = D \Sigma D^T$ also places the vectors in $U$ in radial isotropic position. Now if we change basis so that $A'$ is diagonal (and make the same transformation to the vectors in $U$), we have the desired conclusion. 

Now suppose that $M$ is a diagonal matrix with entries $\lambda_1 \geq \lambda_2 \geq \dots \geq \lambda_d \geq 0$ along the diagonal with the property that it places the vectors in $U$ in radial isotropic position with respect to $c$. Then set
$$W = w_1, w_2, \dots, w_n \mbox{ with } w_i \triangleq \sqrt{\frac{d}{n}} \Big (\frac{M u_i}{\|M u_i\|}\Big )$$
By construction, this is an equal norm Parseval frame.  What remains is to bound the total squared distance between $U$ and $W$. In Lemma~\ref{lem:maintech}, we show that $\mbox{dist}^2(U, W) \leq 8 \epsilon d^2 + 4 \gamma' d^2$, where $\gamma'$ is again another negligible term. Concretely we can choose $\gamma' = \max_i \frac{n}{d}(\|\eta_i \|^2 + 2 \|\eta_i\|)$ when we apply Lemma~\ref{lem:maintech}. 

Finally, for any three vectors $a$, $b$ and $c$ we have the triangle-like inequality $$\| a -c\|_2^2 \leq 2 \Big ( \|a - b\|_2^2 + \|b-c\|_2^2 \Big )$$ and when we apply this for all triples of vectors $v_i$, $u_i$ and $w_i$ we have
$$\mbox{dist}^2(V, W) \leq 2 \mbox{ dist}^2(V, U) + 2 \mbox{ dist}^2(U, W) \leq 20 \epsilon d^2 $$
where the last inequality follows if $\gamma' \leq \epsilon$. This now completes the proof. 
\end{proof}

\begin{lemma}\label{lem:maintech}
With $U$, $M$ and $W$ as defined\footnote{In particular, $U$ is a $4\epsilon$-nearly Parseval frame, $M$ is an entrywise diagonal matrix that puts $U$ in radial isotropic position, and $W$ is defined based on $U$ and $M$.} in the proof of Theorem~\ref{thm:main} and under the condition that for each i $$(1-\gamma')\frac{d}{n}  \leq \|u_i\|_2^2 \leq (1+\gamma')\frac{d}{n} $$ we have that $\mbox{{\em dist}}^2(U, W) \leq 8 \epsilon d^2 + 4 \gamma' d^2$.
\end{lemma}

\begin{proof}
First we introduce a notion of majorization:
\begin{definition}
For $d$ element sequences $x$ and $y$ we say $x \succeq y$ if for all $1 \leq j \leq d$ we have $\sum_{i = 1}^j x_i \geq \sum_{i =1}^j y_i$ and moreover $\sum_{i =1}^d x_i = \sum_{i =1}^d y_i$.
\end{definition}
Next we introduce a notion of distance, similar to the Wasserstein distance, but for vectors that are not necessarily nonnegative:
\begin{definition}
 If $x \succeq y$, we define
$$\mathcal{T}(x, y)  \triangleq \sum_{j =1}^d j (y_j - x_j) = \sum_{j = 1}^d \sum_{i =1}^j x_i - y_i$$
\end{definition}
\noindent We claim that $\|x - y\|_1 \leq \mathcal{T}(x, y)$. What makes this bound subtle is that the vector $x - y$ can (and will) have negative entries. Later on, we will crucially rely on the fact that $\mathcal{T}$ is {\em linear}. But here, we observe that when $x \succeq y$ it is equivalent to the Wasserstein distance $\mathcal{W}(x, y)$ on the real line. In particular
$$\mathcal{T}(x, y) = \sum_{j = 1}^d \Big ( \sum_{i =1}^j x_i \Big ) - \Big ( \sum_{i =1}^j y_i \Big ) = \sum_{j = 1}^d \Big | \Big ( \sum_{i =1}^j x_i \Big ) - \Big ( \sum_{i =1}^j y_i \Big ) \Big |  = \mathcal{W}(x, y)$$
where the last equality follows because the second to last expression is the dual form of the Wasserstein distance (again, on the real line). Finally, because $\mathcal{W}$ measures the transport cost of changing $x$ into $y$ and furthermore each difference between the vectors must be moved distance at least one, we have $\mathcal{W}(x, y) \geq \|x - y\|_1$, which completes the argument.

 Lastly we define some useful sequences of helper vectors. Let $\widetilde{W} = \widetilde{w}_1, \widetilde{w}_2, \dots, \widetilde{w}_n$ with
$$\widetilde{w}_i \triangleq (\|u_i \|) \Big (\frac{M u_i}{\|M u_i\|}\Big )$$
And let $x_i$ be the result of entrywise squaring $u_i$ and let $y_i$ be the result of entrywise squaring $\widetilde{w}_i$. By construction, the sum of the entries in $x_i$ and in $y_i$ are the same. We prove in Claim~\ref{claim:order} that $y_i \succeq x_i$ for each $i$. Now we have
$$\mbox{dist}^2(U, \widetilde{W}) = \sum_{i =1}^n \sum_{j =1}^d \Big ( (u_i)_j - (\widetilde{w}_i)_j \Big )^2 \leq \sum_{i =1}^n \sum_{j =1}^d \Big | (x_i)_j - (y_i)_j \Big | =  \sum_{i =1}^n \| x_i - y_i \|_1 $$
where the first inequality follows because for any real values $a$ and $b$ with the same sign we have $(a-b)^2 \leq |a^2 -b^2|$. Further manipulations give
$$\mbox{dist}^2(U, \widetilde{W}) \leq  \sum_{i =1}^n \mathcal{T}(y_i, x_i) =  \mathcal{T}\Big ( \sum_{i=1}^n y_i, \sum_{i =1}^n x_i\Big) = 2 \sum_{j=1}^d j \Big ( \sum_{i=1}^n ((u_i)_j)^2 - ((\widetilde{w}_i)_j)^2 \Big ) $$
Since $M$ put the vectors $U$ in radial isotropic position with respect to $c$ and the squared norm of each $u_i$ is between $(1-\gamma')\frac{d}{n}$ and $(1+\gamma')\frac{d}{n}$, we have that $\sum_{i =1}^n ((\widetilde{w}_i)_j)^2  \geq 1-\gamma'$. 
And because the vectors $U = u_1, u_2, \dots, u_n$ are a $4 \epsilon$-nearly Parseval frame, we have $\sum_{i =1}^n ((u_i)_j)^2 \leq 1 + 4\epsilon$ which gives
$$\mbox{dist}^2(U, \widetilde{W}) \leq  \sum_{j=1}^d j (4\epsilon + \gamma') \leq 4\epsilon d^2 +  \gamma' d^2 $$
To complete the proof, using the fact that for each $i$, $w_i$ and $\widetilde{w}_i$ differ by a scaling factor whose square is between $1 - \gamma'$ and $1+\gamma'$, we have
$$\mbox{dist}^2(\widetilde{W}, W) \leq 2 \gamma'$$
and putting it all together we get
$$\mbox{dist}^2(U, W) \leq 2\mbox{ dist}^2(U, \widetilde{W}) + 2 \mbox{ dist}^2(\widetilde{W}, W) \leq 8 \epsilon d^2 + 4 \gamma' d^2$$
again using the triangle-like inequality. 
\end{proof}

\begin{claim}\label{claim:order}
With $x_i$ and $y_i$ as defined in Lemma~\ref{lem:maintech}, we have $y_i \succeq x_i$. 
\end{claim}

\begin{proof}
Recall that by construction, $x_i$ and $y_i$ are nonnegative and the sum of their entries is the same. Thus showing that for any $1 \leq j < d$, the inequality
$\sum_{k =1}^j (y_i)_k \geq \sum_{k =1}^j (x_i)_k$ holds follows from showing instead that
$$\frac{\sum_{k =1}^j (y_i)_k}{\sum_{k =j+1}^d (y_i)_k} \geq \frac{\sum_{k =1}^j (x_i)_k}{\sum_{k =j+1}^d (x_i)_k}$$
Now to complete the proof we observe that
\begin{eqnarray*}
\frac{\sum_{k =1}^j (y_i)_k}{\sum_{k =j+1}^d (y_i)_k} &=& \frac{\sum_{k =1}^j  \lambda_k^2 ((u_i)_k)^2}{\sum_{k =j+1}^d \lambda_k^2 ((u_i)_k)^2} \\ 
&\geq& \frac{\sum_{k =1}^j  \lambda_j^2 ((u_i)_k)^2}{\sum_{k =j+1}^d \lambda_j^2 ((u_i)_k)^2} = \frac{\sum_{k =1}^j ((u_i)_k)^2}{\sum_{k =j+1}^d ((u_i)_k)^2} = \frac{\sum_{k =1}^j (x_i)_k}{\sum_{k =j+1}^d (x_i)_k}
\end{eqnarray*}
where the inequality follows from the assumption $\lambda_1 \geq \lambda_2 \geq \dots \geq \lambda_d \geq 0$. 
\end{proof}

\section{An Algorithm for the Paulsen Problem}

Every step of the proof of Theorem~\ref{thm:main} is straightforward to implement algorithmically, except for the step where we compute the transformation $A$ that places the set of vectors $U$ in radial isotropic position. Fortunately, Hardt and Moitra \cite{HM} gave an algorithm for computing $A$ under a slight strengthening of Barthe's conditions which holds in our setting. Informally, they require the vector $c$ to be strictly inside the basis polytope according to the following notion of scaling:

\begin{definition}
Let $(1-\alpha) \mathcal{B}(U)$ denote the set of vectors $c$ with the following properties: $(1)$ $\sum_{i=1}^n c_i = d$ $(2)$ $0 \leq c_i \leq 1$ for all i and $(3)$ for all nonnegative directions $u$ with $u_{min} = 0$, $$(1-\alpha) \max_{v \in \mathcal{B}(U)} u^T v \geq u^T c$$
\end{definition}


We will state a special case of their main theorem, which is sufficient for our purposes. 

\begin{theorem}\cite{HM}\label{thm:radalg}
Let $\delta > 0$ and $\alpha > 0$. Suppose $U = u_1, u_2, \dots, u_n \in \R^d $ has the property that every set of $d$ vectors are linearly independent. Then given $c \in (1-\alpha) \mathcal{B}(U)$, there is an algorithm to find a linear transformation $A$ so that 
$$\sum_{i=1}^n c_i \Big ( \frac{A u_i}{\|A u_i \|} \Big ) \Big (\frac{A u_i}{\|A u_i \|} \Big )^T = I + J$$
where $\|J\|_\infty \leq \delta$ \---- i.e. the largest entry of $J$ in absolute value is at most $\delta$. The running time is polynomial in $1/\alpha$, $\log 1/\delta$ and $L$ where $L$ is an upper bound on the bit complexity of $U$ and $c$. 
\end{theorem}

By combining their algorithm with our proof of Theorem~\ref{thm:main} we get: 

\begin{corollary}
Suppose $V = v_1, v_2, \dots, v_n \in \R^d$ is an $\epsilon$-nearly equal norm Parseval frame. Furthermore suppose $n > d$. Then given $\delta > 0$, there is an algorithm to compute a $\delta$-nearly equal norm Parseval frame $W$ with $$\mbox{{\em dist}}^2(V, W) \leq 20 \epsilon d^2$$
whose running time is polynomial in $\log 1/\delta$ and $L$ where $L$ is an upper bound on the bit complexity of $V$. 
\end{corollary}

\begin{proof}
We perturb $V$ as in the proof of Theorem~\ref{thm:main} and run the algorithm in Theorem~\ref{thm:radalg} on $U$ with $c$ with $c_i = \frac{d}{n}$ for all $i$ and some $\delta'$ to be specified later. First we claim that when every set of $d$ vectors of $u$ are linearly independent and $n > d$, then $c \in (1-\alpha) \mathcal{B}(U)$ for $\alpha = \frac{1}{n}$. Without loss of generality, suppose $u_1 \geq u_2 \geq \cdots \geq u_n = 0$. Then 
$$\max_{v \in \mathcal{B}(U)} u^T v = \sum_{i =1}^d u_i \geq \Big ( \frac{d}{n-1} \Big ) \sum_{i =1}^{n-1} u_i = \Big ( \frac{d}{n-1} \Big ) \sum_{i =1}^n u_i = \Big (\frac{n}{n-1} \Big ) u^T c$$
as desired. Moreover we only needed the perturbation to be polynomially small in $\epsilon$, $1/d$ and $1/n$ and hence we can ensure that its bit complexity is a polynomial in the bit complexity of $V$. 

Finally we can choose $\delta' = \frac{\delta \epsilon}{d^3}$ so that the output is a $\frac{\delta \epsilon}{d}$-nearly equal norm Parseval frame. Finally note that our bound on the squared distance between $V$ and $W$ in Lemma~\ref{lem:maintech} used the fact that $W$ was an equal norm Parseval frame. But it is easy to see that the slack in the bounds we used can accommodate a $\frac{\delta \epsilon}{d}$--nearly equal norm Parseval frame instead. 
\end{proof}

This answers an open question of \cite{KLLR}, where they ask whether there is an algorithm for finding an equal norm Parseval frame up to some precision $\delta$ whose running time is polynomial in $\log 1/\delta$. 

\subsubsection*{Acknowledgements}

We thank Avi Wigderson for numerous helpful comments on an earlier draft.


\begin{thebibliography}{99}

\bibitem{Ba}
K. Ball. 
\newblock Volumes of Sections of Cubes and Related Problems.
\newblock {\em Geometric Aspects of Functional Analysis}, pages 251--260, 1989. 

\bibitem{B}
F. Barthe. 
\newblock On a Reverse Form of the Brascamp-Lieb Inequality. 
\newblock {\em Inventiones Mathematicae}, 134(2):335--361, 1998. 

\bibitem{BCCT}
J. Bennett, A. Carbery, M. Christ and T. Tao. 
\newblock The Brascamp-Lieb Inequalities: Finiteness, Structure, and Extremals. 
\newblock {\em Geometric and Functional Analysis}, 17(5):1343--1415, 2008. 

\bibitem{BC}
B. G. Bodman and P. G. Casazza. 
\newblock The Road to Equal-norm Parseval Frames. 
\newblock {\em Linear Algebra and Its Applications}, 404:118--146, 2005. 

\bibitem{CC}
J. Cahill and P. G. Casazza. 
\newblock The Paulsen Problem in Operator Theory.
\newblock {\em Operators and Matrices}, 7(1):117--130, 2013. 

\bibitem{C}.
P. G. Casazza. 
\newblock The Kadison–Singer and Paulsen Problems in Finite Frame Theory.
\newblock {\em Finite Frames}, pages 381--413. 

\bibitem{CFM}
P. G. Casazza, M. Fickus and D. Mixon. 
\newblock Auto-tuning Unit Norm Frames. 
\newblock {\em Applied and Computational Harmonic Analysis}, 32(1):1--15, 2012. 

\bibitem{CL}
P. G. Casazza and R. G. Lynch. 
\newblock A Brief Introduction to Hilbert Space Frame Theory and Its Applications.
\newblock AMS Short Course, 2015. 

\bibitem{DSW}
Z. Dvir, S. Saraf and A. Wigderson. 
\newblock Superquadratic Lower Bound for $3$-Query Locally Correctable Codes over the Reals. 
\newblock {\em Theory of Computing}, 13(11):1--36, 2017. 

\bibitem{E}
J. Edmonds. 
\newblock Submodular functions, matroids, and certain polyhedra. 
\newblock {\em Combinatorial Structures}, pages 69--87, 1970. 

\bibitem{F}
J. Forster.
\newblock A Linear Lower Bound on the Unbounded Error Probabilistic Community Complexity. 
\newblock {\em JCSS}, 65:612--625, 2002. 

\bibitem{GGOW1}
A. Garg, L. Gurvits, R. Oliveira and A. Wigderson.
\newblock A Deterministic Polynomial Time Algorithm for Non-Commutative Rational Identity Testing. 
\newblock {\em FOCS}, pages 109--117, 2016. 

\bibitem{GGOW}
A. Garg, L. Gurvits, R. Oliveira and A. Wigderson.
\newblock Algorithmic and Optimization Aspects of Brascamp-Lieb Inequalities, via Operator Scaling. 
\newblock {\em STOC}, pages 397--409, 2017. 

\bibitem{G}
L. Gurvits. 
\newblock Classical Complexity and Quantum Entanglement. 
\newblock {JCSS}, 69(3):448--484, 2004. 

\bibitem{HM}
M. Hardt and A. Moitra. 
\newblock Algorithms and Hardness for Robust Subspace Recovery.
\newblock {\em COLT}, pages 354--375, 2013. 

\bibitem{KLLR}
T. C. Kwok, L. C. Lau, Y. T. Lee and A. Ramachandran. 
\newblock The Paulsen Problem, Continuous Operator Scaling, and Smoothed Analysis. 
\newblock {\em STOC}, pages 182--189, 2018. 

\bibitem{M}
D. G. Mixon.
\newblock Four Open Problems in Frame Theory. 
\newblock {\em Frames and Algebraic and Combinatorial Geometry}, 2015.


\end{thebibliography}
\end{document}